\documentclass[12pt,reqno]{amsart}

\numberwithin{equation}{section}

\newtheorem{teo}{Theorem }[section]

\newtheorem{lem}[teo]{Lemma}
\newtheorem{prop}[teo]{Proposition}

\newtheorem{rem}{Remark}
\def \e{e^{2bx}}

\begin{document}


\title[KdV-BURGERS ]
      { THE 2D ZAKHAROV-KUZNETSOV-BURGERS  EQUATION ON A STRIP }
\author[ N.~A. Larkin]{Nikolai A. Larkin$^{\dag}$}

\address
{
Departamento de Matem\'atica, Universidade Estadual
de Maring\'a, Av. Colombo 5790: Ag\^encia UEM, 87020-900, Maring\'a, PR, Brazil
}

\email{$^{\dag}$nlarkine@uem.br}

\thanks{}

\thanks{MSC2010 35Q53;35B35}

\begin{abstract}
An initial-boundary value problem for the  2D
Zakharov-Kuznetsov-Burgers equation posed  on a channel-type strip
was considered. The existence and uniqueness results for regular and
weak  solutions in weighted spaces as well as exponential decay of
small solutions without restrictions on the width of a strip were
proven both for regular solutions in an elevated norm and for weak
solutions in  the $L^2$-norm.
\end{abstract}

\maketitle

\section{Introduction}\label{introduction}

We are concerned with an initial-boundary value problem (IBVP) for
the two-dimensional Zakharov-Kuznetsov-Burgers (ZKB) equation
\begin{equation}\label{kdvb}
u_t+u_x-u_{xx}+uu_x +u_{xxx}+u_{xyy}=0\
\end{equation}
posed on a strip modeling an infinite channel
$\{(x,y)\in\mathbb{R}^2:\ x\in \mathbb{R},\,y \in (0,B), \, B>0\}.$
This equation is a two-dimensional analog of the well-known
Korteweg-de Vries-Burgers (KdV) equation
\begin{equation}\label{kdv}
u_t+u_x-u_{xx}+uu_x+u_{xxx}=0
\end{equation}
 which includes dissipation and dispersion and has been
studied by various researchers due to its applications in Mechanics
and Physics \cite{bona1,bona2,marcelo}.
 One can find  extensive bibliography and sharp results on
 decay rates of solutions to the Cauchy problem (IVP) for (1.2) in
\cite{bona1}. Exponential decay of solutions to the initial problem
for \eqref{kdv} with additional damping has been  established in
\cite{marcelo}. Equations \eqref{kdvb} and \eqref{kdv} are typical
examples of so-called dispersive equations which attract
considerable attention of both pure and applied mathematicians in
the past decades.

 Quite recently, the interest on
dispersive equations became to be extended to  multi-dimensional
models such as Kadomtsev-Petviashvili (KP) and Zakharov-Kuznetsov
(ZK) equations \cite{zk}. As far as the ZK equation and its
generalizations are concerned, the results on IVPs
 can be found in
\cite{fam1,farah,pastor1,pastor2,sautlin,ribaud,temam2} and IBVPs
were studied in \cite{dorlar1,fam2,familark,lartron,lar1,temam2}. In
\cite{lartron,lar1} was shown that IBVP for the ZK equation posed on
a half-strip unbounded in $x$ direction with the Dirichlet
conditions on the boundaries  possesses regular solutions which
decay exponentially as $t\to \infty$ provided initial data are
sufficiently small and the width of a half-strip is not too large.
This means that the ZK equation may create an internal dissipative
mechanism for some types of IBVPs. \par
 The goal of our note is to prove
that the ZKB equation on a strip also may create a dissipative
effect without adding any artificial damping. We must mention that
IBVP for the ZK equation on a strip $(x\in(0,1),\,y\in\mathbb{R})$
has been studied in \cite{dorlar1,temam1} and IBVPs on a strip
$(y\in(0,L),\,x\in \mathbb{R})$  for the ZK equation were considered
in \cite{fambayk} and for the ZK equation with some internal damping
in \cite{fam3}. In the domain $(y\in(0,B),\,x\in\mathbb{R},\,t>0 )$,
the term $u_x$ in \eqref{kdvb} can be  scaled out by a simple change
of variables. Nevertheless, it can not be safely ignored for
problems posed  both on finite and semi-infinite intervals as well
as on infinite in $y$ direction bands without changes in the
original domain \cite{dorlar1,rozan}.

The main results of our paper are the existence  and uniqueness of
regular and weak  global-in-time solutions for  \eqref{kdvb} posed
on a strip with the Dirichlet boundary conditions  and the
exponential decay rate of these solutions as well as continuous
dependence on initial data.

The paper has the following structure. Section 1 is Introduction.
Section \ref{problem} contains formulation of the problem. In
Section \ref{regexist}, we prove  global existence and uniqueness
theorems for regular solutions in some weighted spaces and
continuous dependence on initial data. In Section \ref{regdecay}, we
prove exponential decay of small regular solutions in an elevated
norm corresponding to the $H^1(\mathcal{S})$-norm. In Section
\ref{weak},
 we prove the existence, uniqueness and continuous dependence on initial data for weak solutions as well
  as the exponential decay rate of the $L^2(\mathcal{S})$-norm for small solutions
without limitations on the width of the strip.

\section{Problem and preliminaries}\label{problem}

Let $B,T,r$ be finite positive numbers. Define
$\mathcal{S}=\{(x,y)\in\mathbb{R}^2:\ x\in\mathbb{R},\ y\in(0,B)\};$
$\mathcal{S}_r=\{(x,y)\in\mathbb{R}^2:\ x\in (-r,+\infty),\,
y\in(0,B)\}$ and $\mathcal{S}_T=\mathcal{S}\times (0,T).$

Hereafter subscripts $u_x,\ u_{xy},$ etc. denote the partial
derivatives, as well as $\partial_x$ or $\partial_{xy}^2$ when it is
convenient. Operators $\nabla$ and $\Delta$ are the gradient and
Laplacian acting over $\mathcal{S}.$ By $(\cdot,\cdot)$ and
$\|\cdot\|$ we denote the inner product and the norm in
$L^2(\mathcal{S}),$ and $\|\cdot\|_{H^k}$ stands for norms in the
$L^2$-based Sobolev spaces. We will use also the spaces $H^s\cap
L^2_b$, where $L^2_b=L^2(e^{2bx}dx)$, see \cite{kato}.

Consider the following  IBVP:
\begin{align}
&Lu\equiv u_t-u_{xx}+uu_x+u_{xxx}+u_{xyy}=0,\ \ \text{in}\
\mathcal{S}_T; \label{2.1}
\\
&u(x,0,t)= u(x,B,t)=0,\; x\in \mathbb{R},\ t>0; \label{2.2}
\\
&u(x,y,0)=u_0(x,y),\ \ (x,y)\in\mathcal{S}. \label{2.3}
\end{align}

\section{Existence of regular solutions}\label{regexist}

{\bf Approximate solutions}. We will construct  solutions to
\eqref{2.1}-\eqref{2.3} by the Faedo-Galerkin method: let ${w_j(y)}$
be orthonormal in $L^2(\mathcal{S})$ eigenfunctions of the following
Dirichlet problem:
\begin{align}
&w_{jyy}+\lambda_j w_j=0, \,y\in (0,B); \label{2.4} \\
&w_j(0)=w_j(B)=0. \label{2.5}
\end{align}

Define approximate solutions of \eqref{2.1}-\eqref{2.3} as follows:
\begin{equation}
 u^N(x,y,t)=\sum^N_{j=1} w_j(y)g_j(x,t), \label{UN}
\end{equation}
where $g_j(x,t)$ are solutions to the following Cauchy problem for
the system of $N$ generalized Korteweg-de Vries equations:
\begin{align}
&\frac{\partial}{\partial t}g_j(x,t)+\frac{\partial^3}{\partial
x^3}g_j(x,t)-\frac{\partial^2}{\partial x^2}g_j(x,t)-\lambda_j\frac{\partial}{\partial x} g_{j}(x,t) \nonumber\\
&+\int^B_0 u^N(x,y,t)u^N_x(x,y,t)w_j(y)\, dy
=0, \label{2.6}\\
&g_j(x,0)=\int^B_0 w_j(y)u_0(x,y)\,dy,\;j=1,...,N. \label{2.7}
\end{align}
It is known that for $g_j(x,0)\in H^s,\,s\geq 3,$ the Cauchy problem
\eqref{2.6}-\eqref{2.7} has a unique regular solution $g_j\in
L^{\infty}(0,T;H^s(\mathcal{S})\cap L^2_b(\mathcal{S}))\cap
L^2(0,T;H^{s+1}(\mathcal{S})\cap L^2_b(\mathcal{S}))$
\cite{bona1,kato,ponce}.
 To prove the existence of  global solutions for
 \eqref{2.1}-\eqref{2.3}, we need uniform in $N$  global in $t$
 estimates of approximate solutions $u^N(x,y,t).$\\
 {\bf  Estimate I.} Multiply the j-th equation of \eqref{2.6} by
 $g_j$, sum up over $j=1,...,N$ and integrate the result with respect
 to $x$ over $\mathbb{R}$ to obtain
 $$\frac{d}{dt}\|u^N\|^2(t)+2\|u^N_x\|^2(t)= 0\nonumber$$
 which implies
 \begin{equation}
\|u^N\|^2(t) +2\int_0^t\|u^N_x\|^2(s)\,ds=\|u^N_0\|^2\quad \forall t
\in (0,T). \label{FE}
\end{equation}

It follows from here that for $N$ sufficiently large and $\forall
t>0$

\begin{equation}
\|u^N\|^2(t)+2\int^t_0\ \|u^N_{x}\|^2(s)\,ds =\|u^N\|^2(0)\leq
2\|u_0\|^2. \label{E1}
\end{equation}
In our calculations we will drop the index $N$ where it is not
ambiguous.

{\bf  Estimate II.} For some positive $b$, multiply the j-th
equation of \eqref{2.6} by $e^{2bx}g_j$ , sum up over $j=1,...,N$
and integrate the result with respect  to $x$ over $\mathbb{R}.$
Dropping the index $N$, we get
\begin{align}
&\frac{d}{dt}(e^{2bx},u^2)(t)+(2+6b)(e^{2bx},u^2_x)(t)+2b(e^{2bx},u^2_y)(t)\nonumber\\
&-\frac{4b}{3}(e^{2bx},u^3)(t) -(2b^2+8b^3)(e^{2bx},u^2)(t)=0.
\label{2}
\end{align}

In our calculations, we will frequently use the following
multiplicative inequalities \cite{lady}:
\begin{prop} \label{GN}
i) For all $u \in H^1(\mathbb{R}^2)$
\begin{equation} {\|u\|}_{L^4(\mathbb{R}^2)}^2 \leq 2 {\|u\|}_{L^2(\mathbb{R}^2)}{\|\nabla u\|}_{L^2(\mathbb{R}^2)}.
 \label{p1}
\end{equation}
\qquad \qquad \qquad \qquad ii) For all $u \in H^1(D)$
\begin{equation} {\|u\|}_{L^4(D)}^2 \leq C_D {\|u\|}_{L^2(D)}{\|u\|}_{H^1(D)}, \label{p2}
\end{equation}
where the constant $C_D$ depends on a way of continuation of $u \in
 H^1(D)$ as $ \tilde{u}(\mathbb{R}^2)$ such that $\tilde{u}(D)=u(D).$
\end{prop}

Extending $u^N(x,y,t)$ for a fixed $t$ into exterior of
$\mathcal{S}$ by 0 and exploiting the Gagliardo-Nirenberg inequality
\eqref{p1}, we find
 \begin{equation}
\frac{4b}{3}(e^{2bx}u^3)(t)\leq
b(e^{2bx},u^2_y)(t)\nonumber\\+2b(e^{2bx},u^2_x)(t)+2(b^3+\frac{8b}{9}\|u_0^N\|^2)(e^{2bx},u^2)(t).\label{nl}
\end{equation}
Substituting this into \eqref{2}, we come to the inequality
\begin{align}
&\frac{d}{dt}(e^{2bx},u^2)(t)+(2+4b)(e^{2bx},u^2_x)(t)+b(e^{2bx},u^2_y)(t)\nonumber\\
&\leq C(b)(1+\|u_0\|^2)(e^{2bx},u^2)(t). \label{2.9}
\end{align}
By the Gronwall lemma,
$$(e^{2bx},u^2)(t)\leq C(b,T,\|u_o\|)(e^{2bx},u^2_0).\nonumber$$
Returning to \eqref{2.9} gives
\begin{align}
(e^{2bx},|u^N|^2)(t)+\int_0^t (e^{2bx},|\nabla u^N|^2)(\tau )d
\tau\nonumber\\ \leq C(b,T,\|u_0\|)(e^{2bx},u^2_0)\quad \forall t
\in (0,T). \label{E2}
\end{align}
It follows from this estimate and (3.6) that uniformly in $N$ and
for any $r>0$ and $t \in (0,T)$
\begin{align}
 &\|u^N\|^2(t)+\int_0^t\int_0^B\int_{-r}^{+\infty}
|\nabla u^N|^2\,dx\,dy\,ds\nonumber\\& \leq
\mathbb{C}(r,b,T,\|u_0\|)(e^{2bx},u_0^2), \label{H1}
\end{align}
 where
$\mathbb{C}$ does not depend on $N$.

Estimates \eqref{E2}, \eqref{H1} make it possible to prove the
existence of a weak solution to \eqref{2.1}-\eqref{2.3} passing to
the limit in \eqref{2.6} as $N \to \infty$. For details of passing
to the limit  in the nonlinear term  see \cite{kato}.

{\bf  Estimate III.} Multiplying the j-th equation of \eqref{2.6} by
$-(\e g_{jx})_x$, and dropping the index $N$, we come to the
equality
\begin{align}
&\frac{d}{dt}(\e,
u_x^2)(t)+(2+6b)(\e,u_{xx}^2)(t)+2b(\e,u_{xy}^2)(t)\nonumber\\
&-(4b^2+8b^3)(\e,u_x^2)(t)+(\e,u_x^3)(t)-2b(\e
u,u_x^2)(t)=0.\label{3}
\end{align}
Making use of Proposition \ref{GN}, we estimate
\begin{align}
&I_1=(\e,u_x^3)(t)\leq
\|u_x\|(t)\|e^{bx}u_x\|^2(t)_{L^4(\mathcal{S})}\nonumber\\
&\leq 2\|u_x\|(t)\|e^{bx}u_x\|(t)\|\nabla(e^{bx} u_x)\|(t)\nonumber\\
&\leq \delta(\e,2u_{xx}^2+u_{xy}^2)(t)+2\big[\delta
b^2+\frac{\|u_x\|^2(t)}{2\delta}\big](\e,u_x^2)(t).\nonumber
\end{align}
Similarly,
\begin{align}
&I_2=2b(\e,uu_x^2)(t)\leq \delta(\e,2u_{xx}^2+u_{xy}^2)(t)\nonumber\\
&+\big[2b^2\delta+\frac{4b^2}{\delta}\|u_0\|^2(t)\big](\e,u_x^2)(t).\nonumber
\end{align}
Substituting $I_1,I_2$ into \eqref{3} and taking $2\delta=b$, we
obtain for $\forall t\in(0,T):$
\begin{align}
&(\e,|u^N_x|^2)(t)+\int_0^t(\e,|\nabla
u^N_x|^2)(s)\,ds\nonumber\\&\leq C(b,T,\|u_0\|)(\e,u_{0x}^2).
\label{E3}
\end{align}

{\bf  Estimate IV.} Multiplying the j-th equation of \eqref{2.6} by
$-2(\e\lambda g_{j})$, and dropping the index $N$, we come to the
equality
\begin{align}
&\frac{d}{dt}(\e,
u_y^2)(t)+(2+6b)(\e,u_{xy}^2)(t)+2b(\e,u_{yy}^2)(t)\nonumber\\
&-(4b^2+8b^3)(\e,u_y^2)(t)+2(1-b)(\e,u_x u_y^2)(t)=0.\label{4}
\end{align}
Making use of Proposition \ref{GN}, we estimate
\begin{align}
&I=2(1-b)(\e,u_x u_y^2)(t)\leq
2C_D(1\nonumber\\&+b)\|u_x\|(t)\|e^{bx}u_y\|(t)\|
(e^{bx}u_y)\|_{H^1(\mathcal{S})}(t)\nonumber\\
&\leq\delta(\e, 2u_{xy}^2+u_{yy}^2)(t)+\big[2\delta
(1+b^2)\nonumber\\&+\frac{C_D^2(1+b)^2\|u_x\|^2(t)}{\delta}\big](\e,u_y^2)(t).\nonumber
\end{align}
Taking $\delta=b,$ we transform \eqref{4} into the inequality
\begin{align}
&\frac{d}{dt}(\e,
u_y^2)(t)+(2+4b)(\e,u_{xy}^2)(t)+b(\e,u_{yy}^2)(t)\nonumber\\
&\leq C(b)[1+\|u_x\|(t)^2](\e,u_y^2)(t).\nonumber
\end{align}
Making use of \eqref{E1} and the Gronwall lemma, we get $\;\forall
t\in(0,T):$
\begin{equation}
(\e,|u^N_y|^2)(t)+\int_0^t(\e,|u^N_{yy}|^2)(s)\,ds\nonumber\\\leq
C(b,T,\|u_0\|)(\e,u_{0y}^2). \label{E4}
\end{equation}

This and \eqref{E3} imply that for all finite $r>0$ and all
$t\in(0,T)$
\begin{equation}
\|u^N\|(t)_{H^1(\mathcal{S}_r)}\leq C(r,T,\|u_0\|)(\e,|\nabla
u_0|^2). \label{Ert}
\end{equation}

{\bf  Estimate V.} Multiplying the j-th equation of \eqref{2.6} by
$(\e g_{jxx})_{xx}$, and dropping the index $N$, we come to the
equality
\begin{align}
&\frac{d}{dt}(\e,
u_{xx}^2)(t)+(2+6b)(\e,u_{xxx}^2)(t)+2b(\e,u_{xxy}^2)(t)\nonumber\\
&-(4b^2+8b^3)(\e,u_{xx}^2)(t)-2b(\e,u u_{xx}^2)(t)\nonumber\\&+5(\e
u_x,u_{xx}^2)(t)=0.\label{5}
\end{align}

Using \eqref{p1}, we find
\begin{align}
&I=-2b(\e,u u_{xx}^2)(t)+5(\e u_x,u_{xx}^2)(t)\nonumber\\
&\leq
2\delta(\e,2u_{xxx}^2+u_{xxy}^2)(t)+\big[4b^2\delta+\frac{25}{\delta}\|u_x\|(t)^2\nonumber\\
&+\frac{4b^2}{\delta}\|u\|^2(t)\big](\e,u_{xx}^2)(t).\nonumber
\end{align}
Taking $2\delta=b$ and substituting $I$ into \eqref{5}, we obtain
\begin{align}
&\frac{d}{dt}(\e,
u_{xx}^2)(t)+(2+4b)(\e,u_{xxx}^2)(t)+ b(\e,u_{xxy}^2)(t)\nonumber\\
&\leq C(b)[1+\|u_x\|^2(t)+\|u\|(t)^2](\e,u_{xx}^2)(t).\nonumber
\end{align}
Taking into account \eqref{E1}, we find
\begin{align}
&(\e,|u^N_{xx}|^2)(t)+\int_0^t(\e,|\nabla
u^N_{xx}|^2)(s)\,ds\nonumber\\\leq
&C(b,T,\|u_0\|)(\e,u_{0xx}^2)\quad \forall t\in(0,T).\label{E5}
\end{align}

{\bf  Estimate VI.} Differentiate \eqref{2.6} by $t$ and multiply
the result by $\e g{j_t}$ to obtain

\begin{align}
&\frac{d}{dt}(\e,
u_t^2)(t)+(2+6b)(\e,u_{xt}^2)(t)+2b(\e,u_{ty}^2)(t)\nonumber\\
&-(4b^2+8b^3)(\e,u_{t}^2)(t)+(2-2b)(\e u_x, u_{t}^2)(t)=0.\label{6}
\end{align}
 Making use of \eqref{p1}, we estimate
 \begin{align}
 &I=(2-2b)(\e u_x, u_{t}^2)(t)\leq
 2(2+2b)\|u_x\|(t)\|e^{bx}u_t\|(t)\|\nabla(e^{bx}
 u_t)\|(t)\nonumber\\
 &\delta(\e,2u_{xt}^2+u_{ty}^2)(t)+\big[2b^2\delta+\frac{(2+2b)^2\|u_x\|(t)^2}{\delta}\big](\e,u_t^2)(t).\nonumber
 \end{align}
 Taking $\delta=b$ and substituting $I$ into \eqref{6}, we obtain

\begin{align}
&\frac{d}{dt}(\e,
u_t^2)(t)+(2+4b)(\e,u_{xt}^2)(t)+b(\e,u_{ty}^2)(t)\nonumber\\
&\leq C(b)[1+\|u_x\|(t)^2](\e , u_{t}^2)(t).\nonumber
\end{align}

This implies \quad $\forall t\in 0,T)$:
\begin{align}
(\e,|u^N_t|^2)(t)+\int_0^t(\e,|\nabla
u^N_{s}|^2(s)\,ds\nonumber\\\leq C(b,T,\|u_0\|)(\e,u_{t}^2)(0)\leq
C(b,T,\|u_0\|)\|)J_0,\label{E6}
\end{align}
where
$$J_0=\|u_0\|^2+(\e,u^2_0+|\nabla u_0|^2+|\nabla u_{0x}|^2+u^2_0
u^2_{0x}+|\Delta u_{0x}|^2).$$

{\bf  Estimate VII.} Multiplying the j-th equation of \eqref{2.6} by
$-\e g_{jx}$, we come, dropping the index $N$, to the equality
\begin{align}
&(\e,[u_{xy}^2+u_{xx}^2])(t)=-(\e[u_{t}-(1+2b)u_{xx}],u_x)(t)\nonumber\\&+(\e,u
u^2_x)(t).\label{e7}
\end{align}
Making use of \eqref{p1}, we estimate
\begin{align}
&I=(\e,uu_x^2)(t)\leq \delta(\e,2u_{xx}^2+u_{xy}^2)(t)+
\big[2b^2\delta+\frac{\|u_0\|^2}{\delta}\big](\e,u_x^2)(t).\nonumber
\end{align}
Taking $4\delta=1$, using \eqref{E3}-\eqref{E6} and substituting $I$
into \eqref{e7}, we get

\begin{equation}
(\e,{u^N_{xx}}^2+{u^N_{xy}}^2)(t)\leq C(b,T,\|u_0\|)J_0 \quad
\forall t\in(0,T). \label{E7}
\end{equation}

{\bf  Estimate VIII.} We will need the following lemma :
\begin{lem} \label{supr}
Let $u(x,y): \mathcal{S}\to \mathbb{R}$ be such that
$$ \int_{\mathcal{S}}\e [u^2(x,y)+|\nabla
u(x,y)|^2+u_{xy}^2(x,y)]\,dxdy < \infty$$ and for all
$x\in\mathbb{R}$ there is some $y_0\in [0,B]$ such that
$u(x,y_0)=0.$ Then
\begin{align}
&\sup_{\mathcal{S}}|e^{bx}u(x,y,t)|^2 \leq
\delta(1+2b^2)(\e,u_y^2)(t)+
2\delta(\e,u_{xy}^2)(t)\nonumber\\
&+\frac{2\delta_1}{\delta}(\e,
u_x^2)(t)+\frac{1}{\delta}\big[\frac{1}{\delta_1}+2\delta_1
b^2\big](\e,u^2)(t),\label{esup}
\end{align}
where $\delta, \delta_1$ are arbitrary positive numbers.
\end{lem}

\begin{proof} Denote $v=e^{bx}u.$ Then simple calculations give
\begin{align}
&\sup_{\mathcal{S}} v^2(x,y,t)\leq
\delta[\|v_y\|^2(t)+\|v_{xy}\|^2(t)] +\frac{1}{\delta}[\|v_x\|^2(t)+
\|v\|^2(t)].\nonumber
\end{align}
Returning to the function $u(x,y,t)$, we prove Lemma \ref{supr}
\end{proof}

Multiplying the j-th equation of \eqref{2.6} by $\e g_{jxxx}$, we
come, dropping the index $N$, to the equality
\begin{align}
&(\e,
u_{xxy}^2+u_{xxx}^2)(t)=-(\e[u_t-u_{xx}],u_{xxx})(t)\nonumber\\&-(\e
uu_x,u_{xxx})(t)+2b^2(\e,u_{xy}^2)(t).\label{8}
\end{align}

Using Lemma \ref{supr} and \eqref{E1}, we estimate
\begin{align}
&I=(\e uu_x,u_{xxx})(t)\leq
\|u\|(t)\sup_{\mathcal{S}}|e^{bx}u_x(x,y,t)|\|e^{bx}u_{xxx}\|(t)\nonumber\\
&\leq
\epsilon\|u_0\|^2(\e,u_{xxx}^2)(t)+\frac{1}{4\epsilon}\big[\frac{1}{\delta}(1+2b^2)(\e,u_x^2)(t)\nonumber\\
&+\frac{2}{\delta}(\e
u_{xx}^2)(t)+\delta(1+2b^2)(\e,u_{xy}^2)(t)+2\delta(\e,u_{xxy}^2)(t)\big].
\end{align}

Taking $\epsilon$ and $\delta$ sufficiently small, positive and
substituting $I$ into \eqref{8}, we find
\begin{equation}
(\e, |\nabla u^N_{xx}|^2)(t)\leq C(b,T,\|u_0\|)J_0 \quad \forall
t\in(0,T). \label{E81}
\end{equation}

Consequently, it follows from the equality
$$ -(\e [u^N_t-u^N_{xx}+u^N_{xxx}+u^N_{xyy}+u^N
u^N_x],u^N_{yy})(t)=0
$$
and from

$$ (\e [u^N_t-u^N_{xx}+u^N_{xxx}+u^N_{xyy}+u^N
u^N_x],u^N_{xyy})(t)=0
$$
 that
\begin{equation}
(\e, |u^N_{yy}|^2+|u^N_{xyy}|^2)(t)\leq C(b,T,\|u_0\|)J_0 \quad
\forall t\in(0,T). \label{E82}
\end{equation}

Jointly, estimates \eqref{E3},\eqref{E4}, \eqref{E5},
\eqref{E7},\eqref{E81}, \eqref{E82} read
\begin{align}
&(\e, |u^N|^2+|\nabla u^N|^2 +|\nabla u^N_x|^2+|\nabla
u^N_y|^2+|\nabla u_{xx}^N|^2)(t)\nonumber\\&\leq
C(b,T,\|u_0\|)J_0\quad \forall t\in(0,T). \label{E83}
\end{align}
In other words,
\begin{equation}
e^{bx}u^N,\quad e^{bx}u^N_x\in L^{\infty}(0,T;H^2(\mathcal{S}))
\label{E84}
\end{equation}
and these inclusions are uniform in $N$.

{\bf  Estimate IX.} Differentiating the j-th equation of \eqref{2.6}
with respect to $x$ and multiplying the result by  $\e
\partial_x^4 g_j$, we come, dropping the index $N$, to the equality
\begin{align}
&(\e,
u_{xxxy}^2+u_{xxxx}^2)(t)=+2b^2(\e,u_{xxy}^2)(t)-(\e[u_{xt}-\partial_x^3
u],u_{xxxx})t)\nonumber\\&-(\e[u_x^2+uu_{xx}],\partial_x^4
u)(t).\label{e9}
\end{align}

Making use of Lemma \ref{supr} and \eqref{E83}, we estimate
\begin{align}
&I_1= (\e,u_x^2,\partial_x^4 u)(t)\leq
\|u_x\|(t)\|e^{bx}\partial_x^4
u\|(t)\sup_{\mathcal{S}}|e^{bx}u_x(x,y,t)|\nonumber\\
&\leq \frac{\epsilon_1}{2}(\e,|\partial_x^4
u|^2)(t)+\frac{1}{2\epsilon_1}\|u_x\|^2(t)\big[(1+2b^2)(\e,u_x^2)(t)\nonumber\\
&+2(\e,u_{xx}^2)(t)+(1+2b^2)(\e,u_{xy}^2)(t)+2(\e,u_{xxy}^2)(t)\big]\nonumber\\
&\leq \frac{\epsilon_1}{2}(\e,|\partial_x^4
u|^2)(t)+\frac{1}{2\epsilon_1}C(b,T,\|u_0\|)J_0,\nonumber
\end{align}

\begin{align}
&I_2=(\e u,u_{xx}\partial_x^4 u)(t)\leq\|e^{bx}\partial_x^4
u\|(t)\|u\|(t)\sup_{\mathcal{S}}|e^{bx}u_{xx}(x,y,t|\nonumber\\
&\leq\frac{\epsilon_1}{2}\|u_0\|^2(t)(\e,|\partial_x^4
u|^2)(t)+\frac{1}{2\epsilon_1}\{2\delta(\e,u_{xxxy}^2)(t)\nonumber\\
&+\delta(1+2b^2)(\e,u_{xxy}^2)(t)+\frac{2}{\delta}(\e,u_{xxx}^2)(t)\nonumber\\
&+\frac{1}{\delta}(1+2b^2)(\e,u_{xx}^2)(t)\}.
\end{align}

Applying the Young inequality, taking $\epsilon_1,\;\delta$
sufficiently small positive, substituting $I_1,I_2$ into \eqref{e9}
and integrating the result, we come to the following inequality:
\begin{equation}
\int_0^t(\e, |u^N_{xxxy}|^2+|u^N_{xxxx}|^2)(s)\,ds\leq
C(b,T,\|u_0\|)J_0 \:\;\forall t\in(0,T). \label{E9}
\end{equation}

{\bf  Estimate X.} Multiplying the j-th equation of \eqref{2.6} by
$-\e\lambda^2 g_{jx}$, we come, dropping the index $N$, to the equality
\begin{align}
&(\e,
u_{xxyy}^2+u_{xyyy}^2)(t)=-(\e,u_{ty},u_{xyyy}^2)(t)+(b+2b^2)(\e,u_{xyy}^2)(t)\nonumber\\&-(\e
u_y u_x,u_{xyyy})(t)+(\e uu_{xy},u_{xyyy})(t).\label{e10}
\end{align}
We estimate
\begin{align}
& I_1=-(\e,u_{ty},u_{xyyy})(t)\leq
\frac{\epsilon}{2}(\e,u_{xyyy}^2)(t)+\frac{1}{2\epsilon}(\e,u_{yt}^2)(t),\nonumber\\
&I_2=(\e u_y
u_x,u_{xyyy})(t)\leq\|u_x\|(t)\|e^{bx}u_{xyyy}\|(t)\sup_{\mathcal{S}}|e^{bx}u_y(x,y,t|\nonumber\\
&\leq
\frac{\epsilon}{2}(\e,u_{xyyy}^2)(t)+\frac{\|u_x\|(t)^2}{2\epsilon}\big[(1+2b^2)(\e,u_y^2)(t)\nonumber\\
&+2(\e,u_{xy}^2)(t)+(1+2b^2)(\e,u_{yy}^2)(t)
+2(\e,u_{xyy}^2)(t)\big],\nonumber\\
&I_3=(\e
uu_{xy},u_{xyyy})(t)\leq\|u\|(t)\|e^{bx}u_{xyyy})\|(t)\sup_{\mathcal{S}}|e^{bx}u_{xy}(x,y,t|\nonumber\\
&\leq
\frac{\|u_0\|^2\epsilon_1}{2}(\e,u_{xyyy}^2)(t)+\frac{1}{2\epsilon_1}\big[ 2\delta(\e, u_{xxyy}^2)(t)\nonumber\\
&+\frac{2}{\delta}(\e,u_{xxy}^2)(t)+\delta(1+2b^2)(\e,u_{xyy}^2)(t)\nonumber\\
&+\frac{1}{\delta}(1+2b^2)(\e,u_{xy}^2)(t)\big].\nonumber
\end{align}
Choosing $\epsilon,\;\epsilon_1,\; \delta$ sufficiently small,
positive, after integration, we transform \eqref{e10}  into the form
\begin{equation}
\int_0^T(\e,[|u^N_{xxyy}|^2+|u^N_{xyyy}|^2])(t)\,dt\leq
C(b,T,\|u_0\|)J_0. \label{E101}
\end{equation}
Acting similarly, we get from the scalar product
$$(\e\big[u^N_t-u^N_{xx}+u^N_{xxx}+u^N_{xyy}+u^N u^N_x\big],u^N_{yyyy})(t)=0$$
the estimate
\begin{equation}
\int_0^T(\e,|u^N_{yyy}|^2)(t)\,dt \leq C(b,T,\|u_0\|)J_0.
\label{E102}
\end{equation}

 Estimates \eqref{E83}, \eqref{E84}, \eqref{E9}, \eqref{E101},
\eqref{E102} guarantee that
\begin{equation}
e^{bx}u^N,\quad e^{bx}u^N_x\in L^{\infty}(0,T;H^2(\mathcal{S})\cap
L^2(0,T;H^3(\mathcal{S}))\label{EN}
\end{equation}
and these inclusions do not depend on $N.$ Independence of Estimates
\eqref{E1},\eqref{EN} of $N$ allow us to pass to the limit in
\eqref{2.6} and to prove the following result:
\begin{teo} \label{regsol}
Let $u_0(x,y):\mathbb{R}^2\to \mathbb{R}$ be such that
$u_0(x,0)=u_0(x,B)=0$ and for some $b>0$

$$J_0= \int_{\mathcal{S}}\{u_0^2+\e[u_0^2+|\nabla u_0|^2+|\nabla
u_{0x}|^2+u_0^2 u_{0x}^2 + |\Delta u_{0x}|^2]\}\,dxdy < \infty.$$
 Then there exists a regular solution to
\eqref{2.1}-\eqref{2.3} $u(x,y,t):$
\begin{align}
& u\in L^{\infty}(0,T;L^2(\mathcal{S})),\quad u_x\in
 L^2(0,T;L^2(\mathcal{S}))\nonumber\\
 &e^{bx}u,\;e^{bx}u_x \in L^{\infty}(0,T;H^2(\mathcal{S}))\cap
 L^2(0,T;H^3(\mathcal{S}))\nonumber\\
 &e^{bx}u_t \in L^{\infty}(0,T;(L^2(\mathcal{S})))\cap
 L^2(0,T;H^1(\mathcal{S}))\nonumber
 \end{align}
 which for $a.e. \;t\in(0,T)$ satisfies the identity
\begin{equation}
(e^{bx}\big[u_t-u_{xx}+u_{xxx}+uu_x+u_{xyy}\big]\phi(x,y))(t)=0,
\label{regularsol}
\end{equation}
where $\phi(x,y)$ is an arbitrary function from $L^2(\mathcal{S}).$
\end{teo}
\begin{proof}
Rewrite \eqref{2.6} in the form
\begin{align}
(e^{bx}\big[u^N_t-u^N_{xx}+u^N
u^N_x+u^N_{xxx}+u^N_{xyy}\big]\Phi^N(y)\Psi(x))(t)=0, \label{EqN}
\end{align}
where $\Phi^N(y)$ is an arbitrary function from the set of linear
combinations $\sum_{i=1}^N \alpha_i w_i(y)$ and $\Psi(x)$ is an
arbitrary function from $H^1(\mathbb{R})$. Taking into account
estimates \eqref{E1}, \eqref{EN} and fixing $\Phi^N$, we can easily
pass to the limit as $N\to\infty$ in linear terms of \eqref{EqN}. To
pass to the limit in the nonlinear term, we must use \eqref{Ert} and
repeat arguments of \cite{kato}. Since linear combinations
$[\sum_{i=1}^N \alpha_i w_i(y)]\Psi(x)$ are dense in
$L^2(\mathcal{S}),$ we come to \eqref{regularsol}. This proves the
existence of regular solutions to (2.1)-(2.3).
\end{proof}

\begin{rem} \label{limit estimates}
Estimates \eqref{E1},\eqref{EN} are valid also for  the limit
function $u(x,y,t)$ and \eqref{E1} obtains its sharp form:
\begin{equation}
\|u\|(t)^2+2\int_0^t \|u_x\|(s)^2\,ds=\|u_0\|^2 \quad \forall
t\in(0,T). \label{E1l}
\end{equation}
\end{rem}

{\bf  Uniqueness of a regular solution.}
\begin{teo} \label{uniq} A
regular solution from Theorem \ref{regsol} is uniquely defined.
\end{teo}
\begin{proof}
Let $u_1,\,u_2$ be two distinct regular solutions of
\eqref{2.1}-\eqref{2.3}, then $z=u_1-u_2$ satisfies the following
initial-boundary value problem:
\begin{align}
&z_t-z_{xx}+z_{xxx}+z_{xyy} +\frac{1}{2}(u_1^2-u_2^2)_x=0
\;\mbox{in} \;\mathcal{S}_T,
\label{u1}\\
&z(x,0,t)=z(x,B,t)=0,\quad x\in \mathbb{R},\quad t>0,\label{u2}\\
&z(x,y,0)=0.\quad (x,y)\in \mathcal{S}. \label{u3}
\end{align}
Multiplying \eqref{u1} by $2e^{bx}z$, we get
\begin{align}
&\frac{d}{dt}(\e,z^2)(t)+(2+6b)(\e,z_x^2)(t)-(8b^3+4b^2)(\e,z^2)(t)\nonumber\\
&+2b(\e, z_y^2)(t)+(\e
[u_{1x}+u_{2x}],z^2)(t)\nonumber\\&-b(\e(u_1+u_2),z^2)(t)=0.
\label{u4}
\end{align}

We estimate
\begin{align}
&I_1=(\e(u_{1x}+u_{2x}),z^2)(t)\leq
\|u_{1x}+u_{2x}\|(t)\|e^{bx}z\|(t)^2_{L^4(\mathcal{S})}\nonumber\\&\leq
2\|u_{1x}+u_{2x}\|(t)\|e^{bx}z\|(t)\|\nabla(e^{bx}z)\|(t)\nonumber\\
&\leq
\delta(\e,[2{z_x}^2+{z_y}^2])(t)+[2{b^2}\delta+\frac{2}{\delta}(\|u_{1x}\|^2(t)\nonumber\\
&+\|u_{2x}\|^2(t))](\e,z^2)(t),\nonumber\\
&I_2=b(\e(u_{1}+u_{2}),z^2)(t)\leq
b\|u_{1}+u_{2}\|(t)\|e^{bx}z\|^2_{L^4(\mathcal{S})}\nonumber\\&\leq
2b\|u_{1}+u_{2}\|(t)\|e^{bx}z\|(t)\|\nabla(e^{bx}z)\|(t)\nonumber\\
&\leq
\delta(\e,2z_x^2+z_y^2)(t)+[2b^2\delta+\frac{2b^2}{\delta}(\|u_{1}\|^2(t)
+\|u_{2}\|^2(t))](\e,z^2)(t).\nonumber
\end{align}
Substituting $I_1,I_2$ into \eqref{u4} and taking $\delta>0$
sufficiently small, we find

\begin{align}
&\frac{d}{dt}(\e,z^2)(t)+(2+2b)(\e,z_x^2)(t)+b(\e,z_y^2)(t)\leq
C(b)\big[1+\|u_1\|(t)^2\nonumber\\&+\|u_2\|(t)^2+\|u_{1x}\|(t)^2+\|u_{2x}\|(t)^2\big](\e,z^2)(t).\label{u5}
\end{align}

Since
$$u_i\in L^{\infty}(0,T;L^2(\mathcal{S})),\quad u_{ix}\in
L^2(0,T;L^2(\mathcal{S}))\quad i=1,2,
$$
 then by the Gronwall lemma,
$$(\e, z^2)(t)=0 \quad\forall \:t\in(0,T).$$
Hence, $u_1=u_2 \quad a.e.$ in $\mathcal{S}_T.$
\end{proof}
\begin{rem}
Changing initial condition \eqref{u3} for $z(x,y,0)=z_0(x,y)\ne 0,$
and repeating the proof of Theorem 3.4, we obtain from \eqref{u5}
that
$$(\e,z^2)(t)\leq C(b,T,\|u_0\|)(\e,z^2_0)\quad \forall t\in(0,T).$$
This means continuous dependence of regular solutions on initial
data.
\end{rem}

\section{Decay of regular solutions}\label{regdecay}
 In this section we will prove exponential decay of regular
 solutions in an elevated weighted norm corresponding to the
 $H^1(\mathcal{S})$ norm. We start with  Theorem  \ref{decay1} which
 is crucial for the main result.

\begin{teo} \label{decay1}
Let  $b\in(0, \frac{1}{5}[-1+\sqrt{1+\frac{5\pi^2}{4B^2}}]),\;
\|u_0\|\leq \frac{3\pi}{8B}$ and $u(x,y,t)$ be a regular solution of
\eqref{2.1}-\eqref{2.3}. Then for all finite $B>0$ the following
inequality is true:
\begin{equation}
 \|e^{bx}u\|^2(t)\leq e^{-\chi t}\|e^{bx}u_0\|^2(0),
\label{decay}
\end{equation}
 where $\chi=\frac{1}{20}[-1+\sqrt{1+\frac{5\pi^2}{4B^2}}]\frac{\pi^2}{B^2}$.
\end{teo}
\begin{proof}
Multiplying \eqref{2.1} by $2\e u$, we get the equality
\begin{align}
&\frac{d}{dt}(e^{2bx},u^2)(t)+(2+6b)(e^{2bx},u^2_x)(t)+2b(e^{2bx},u^2_y)(t)\nonumber\\
&-\frac{4b}{3}(e^{2bx},u^3)(t) -(4b^2+8b^3)(e^{2bx},u^2)(t)=0.
\label{d1}
\end{align}
Taking into account \eqref{GN}, we estimate
\begin{align}
&I=\frac{4b}{3}(e^{2bx},u^3)(t)\leq b(e^{2bx},u_y^2+2u_x^2+2b^2
u^2)(t)\nonumber\\&+\frac{16b}{9}\|u_0\|^2(e^{2bx},u^2)(t).\nonumber
\end{align}

The following proposition is principal for our proof.
\begin{prop}\label{propcr}
\begin{equation}
\int_{\mathbb{R}}\int_0^B e^{2bx}u^2(x,y,t)\,dy\,dx\leq
\frac{B^2}{\pi^2}\int_{\mathbb{R}}\int_0^B
e^{2bx}u^2_y(x,y,t)\,dy\,dx. \label{mineq}
\end{equation}

\end{prop}
\begin{proof}
Since $u(x,0,t)=u(x,B,t)=0,$ fixing $(x,t)$, we can use with respect
to $y$ the following Steklov inequality: if $f(y)\in H^1_0(0,\pi)$
then
$$\int_0^{\pi}f^2(y)\,dy\leq \int_0^{\pi} |f_y(y)|^2\,dy.$$
After a corresponding process of scaling we prove  Proposition
\ref{propcr}.
\end{proof}

 Making use of \eqref{mineq} and substituting $I$ into \eqref{d1},
 we come to the following inequality
\begin{align}
&\frac{d}{dt}(e^{2bx},u^2)(t)+(2+4b)(e^{2bx},u^2_x)(t)\nonumber\\
&+\big[\frac{b\pi^2}{B^2}-
4b^2-10b^3-\frac{16b}{9}\|u_0\|^2\big](\e,u^2)(t)\leq 0 \nonumber
\end{align}
which can be rewritten as

\begin{equation} \frac{d}{dt}(e^{2bx},u^2)(t) +\chi
(e^{2bx},u^2)(t)\leq 0, \label{d2}
\end{equation}
where
$$\chi=b\big[\frac{\pi^2}{B^2}-4b-10b^2-\frac{16\|u_0\|^2}{9}\big].$$
Since we need $\chi>0,$ define
\begin{equation}
4b+10b^2=\gamma\frac{\pi^2}{B^2},\quad
\frac{16\|u_0\|^2}{9}=(1-\gamma)^2\frac{\pi^2}{B^2}, \label{algebra}
\end{equation}
where $\gamma\in(0,1).$ It implies
$\chi=bA(\gamma)\frac{\pi^2}{B^2}$ with
$A(\gamma)=\gamma(1-\gamma).$\\
It is easy to see that
$$ \sup_{\gamma\in(0,1)}A(\gamma)=A(\frac{1}{2})=\frac{1}{4}.$$
Solving \eqref{algebra}, we find
$$b=\frac{1}{5}[-1+\sqrt{1+\frac{5\pi^2}{4B^2}}],\quad \|u_0\|\leq
\frac{3\pi}{8B},\quad \chi=b\frac{\pi^2}{4B^2},$$ and from
\eqref{d2} we get
$$ (e^{2bx},u^2)(t)\leq e^{-\chi
t}(e^{2bx},|u_0|^2).$$ The last inequality
 implies \eqref{decay}. The proof of Theorem \ref{decay1} is complete.
 \end{proof}
 Observe that differently from \cite{lartron,lar1}, we do not have
 any restrictions on the width of a strip $B$.

The main result of this section is the following assertion.
\begin{teo} \label{decay2}
Let all the conditions of Theorem \ref{decay1} be fulfilled. Then
regular solutions of \eqref{2.1}-\eqref{2.3} satisfy the following
inequality:
\begin{align}
&(\e,u^2+|\nabla u|^2)(t)\leq C(b,\chi,\|u_0\|)(1+t)e^{-\chi t}(\e,
\big[u_0^2\nonumber\\
&+|u_0|^3 +|\nabla u_0|^2\big]) \label{decaymain}
\end{align}
or
$$\|e^{bx}u\|^2_{H^1(\mathcal{S})}(t)\leq C(b,\chi,\|u_0\|)(1+t)e^{-\chi
t}(\e,
u_0^2+|u_0|^3 +|\nabla u_0|^2).$$
\end{teo}
\begin{proof} We start with the following lemma.
\begin{lem} \label{lem1}
Regular solutions of (2.1)- (2.3) satisfy the following equality:
\begin{align} &e^{\chi t}(\e, |\nabla u|^2)(t)+2\int_0^t
e^{\chi
s}\{(1+3b)(\e,u_{xx}^2)(s)+(1+4b)(\e, u_{xy}^2)(s)\nonumber\\
&+b(\e,u_{yy}^2)(s)+\frac{b}{2}(\e,u^4)(s)\}\,ds=\frac{e^{\chi
t}}{3}(\e,u^3)(t)\nonumber\\
&+\int_0^t e^{\chi s}(\chi+4b^2+8b^3)(\e,|\nabla
u|^2)(s)\,ds+2\int_0^t e^{\chi s}\{(1+4b)(\e u,u_x^2)(s)\nonumber\\
&+4b(\e,uu_y^2)(s)-(\frac{4b^2+8b^3}{3}-\chi)(\e,u^3)(s)\}\,ds\nonumber\\
&+(\e,|\nabla u_0|^2-\frac{u_0^3}{3}). \label{e1l1}
\end{align}
\end{lem}
\begin{proof}
First we transform the scalar product
\begin{align}
&-(e^{bx}\big[u_t-u_{xx}+u_{xxx}+u_{xyy}+uu_x\big],\nonumber\\
&\big[2(e^{bx}u_x)_x+2e^{bx}u_{yy}+e^{bx}u^2\big])(t)=0 \label{e2l1}
\end{align}
into the following equality:
\begin{align}
&\frac{d}{dt}(\e, |\nabla
u|^2-\frac{u^3}{3})(t)+2(1+3b)(\e,u_{xx}^2)(t)\nonumber\\
&+2b(\e,u_{yy}^2)(t)+2(1+4b)(\e,u_{xy}^2)(t)+\frac{b}{2}(\e,u^4)(t)\nonumber\\
&=4b^2(1+2b)(\e,|\nabla
u|^2)(t)-\frac{4b^2(1+2b)}{3}(\e,u^3)(t)\nonumber\\
&+4b(\e,uu_y^2)(t)+2(1+4b)(\e,uu_x^2)(t). \label{e3l1}
\end{align}
To prove \eqref{e3l1}, we estimate separate terms in \eqref{e2l1} as
follows:
\begin{align}
&I_1=-2(e^{bx}\big[u_t-u_{xx}+u_{xxx}+u_{xyy}+uu_x\big],(e^{bx}u_x)_x)(t)\nonumber\\
&=2(e^{2bx}\big[u_t-u_{xx}+u_{xxx}+u_{xyy}+uu_x\big]_x,u_x)(t)\nonumber\\
&=\frac{d}{dt}(\e,u_x^2)(t)+2(1+3b)(\e,u_{xx}^2)(t)+2b(\e,u_{xy}^2)(t)\nonumber\\
&-4b^2(1+2b)(\e,u_x^2)(t)+(\e
u^2,u_{xxx})(t)\nonumber\\
&-8b(\e,uu_x^2)(t)+\frac{8b^3}{3}(\e,u^3)(t),
\nonumber\\
&I_2=-2(e^{bx}\big[u_t-u_{xx}+u_{xxx}+u_{xyy}+uu_x\big],e^{bx}u_{yy})(t)\nonumber\\
&=2(e^{bx}\big[u_t-u_{xx}+u_{xxx}+u_{xyy}+uu_x\big]_y,e^{bx}u_y)(t)\nonumber\\
&=\frac{d}{dt}(\e,u_y^2)(t)+2(1+3b)(\e,u_{xy}^2)(t)+2b(\e,u_{yy}^2)(t)\nonumber\\
&-4b^2(1+2b)(\e,u_y^2)(t)+(\e
u,u_{xyy})(t)-4b(\e,uu_y^2)(t),\nonumber\\
&I_3=-(e^{bx}\big[u_t-u_{xx}+u_{xxx}+u_{xyy}+uu_x\big],e^{bx}u^2))(t)\nonumber\\
&-\frac{d}{dt}(\e,\frac{u^3}{3})(t)+\frac{4b^2}{3}(\e,u^3)(t)+\frac{b}{2}(\e,u^4)(t)\nonumber\\
&-2(\e,uu_x^2)(t)-(\e,u_{xxx}+u_{xyy})(t).\nonumber
\end{align}

Summing $I_1+I_2+I_3$, we obtain \eqref{e3l1}. In turn, multiplying
it by $e^{\chi t}$ and integrating the result over $(0,t)$, we come
to \eqref{e1l1}. The proof of Lemma \ref{lem1} is complete.
\end{proof}
Making use of  \eqref{p1}, we estimate
\begin{align}\
&I_4=\frac{e^{\chi t}}{3}(\e,u^3)(t)\leq \frac{2e^{\chi
t}}{3}\|u_0\|\|e^{bx}u\|(t)\|\nabla(e^{bx}u)\|(t)\nonumber\\
&\leq \frac{e^{\chi t}}{2}\{(\e,|\nabla
u|^2)(t)+[\frac{b^2}{2}+\frac{4\|u_0\|^2}{9}](\e,u^2)(t)\}.\nonumber
\end{align}
Substituting $I_4$ into \eqref{e1l1}, we get
\begin{align} &e^{\chi t}(\e, |\nabla u|^2)(t)+4\int_0^te^{\chi
s}\{(1+3b)(\e,u_{xx}^2)(s)+(1+4b)(\e, u_{xy}^2)(s)\nonumber\\
&+b(\e,u_{yy}^2)(s)\}\,ds\leq
2\int_0^t e^{\chi s}(\chi+\frac{4b^2+8b^3}{3})(\e,u^3)(s)\,ds\nonumber\\
&+2\int_0^t e^{\chi s}\{2(1+4b)(\e u,u_x^2)(s)+4b(\e,uu_y^2)(s)\}\,ds\nonumber\\
&+2\int_0^t e^{\chi s}(\chi+4b^2+8b^3)(\e,|\nabla
u|^2)(s)\,ds\nonumber\\
&+\big[b^2+\frac{8\|u_0\|^2}{9}\big]e^{\chi t}(\e,u^2)(t)
+2(\e,|\nabla u_0|^2+\frac{|u_0|^3}{3}). \label{e4l1}
\end{align}

In order to estimate the right-hand side of \eqref{e4l1}, we will
need the following
\begin{prop} \label{base}
Let Theorem \ref{decay1} be true. Then
\begin{align}
&e^{\chi t}(\e,u^2)(t)+\int_0^t e^{\chi s}(\e,|\nabla
u|^2)(s)\,ds\nonumber\\
&\leq C(b,\chi,\|u_0\|)(1+t)(\e,u_0^2). \label{e5l1}
\end{align}
\end{prop}
\begin{proof}
Consider the equality
$$\int_0^t 2e^{\chi s}(\e
[u_s-u_{xx}+u_{xxx}+u_{xyy}+uu_x],u)(s)\,ds=0$$

which we rewrite as
\begin{align}
&e^{\chi t}(\e,u^2)(t)+2\int_0^t e^{\chi
s}\{(1+3b)(\e,u_x^2)(s)+b(\e,u_y^2)(s)\},ds\nonumber\\
&=\int_0^t e^{\chi s}\frac{4b}{3}(\e,u^3)(s)+\int_0^t e^{\chi
s}(\chi+4b^2+8b^3)(\e,u^2)(s)\,ds\nonumber\\
&+(\e,u_0^2). \label{e6l1}
\end{align}
By Proposition \ref{GN}, we estimate
\begin{align}
&I_1=\frac{4b}{3}(\e,u^3)(t)\leq
\frac{8b}{3}\|u\|(t)\|e^{bx}u\|)(t)\|\nabla(e^{bx}u)\|(t)\nonumber\\
&\leq b(\e,
2u_x^2+u_y^2)(t)+[2b^3+\frac{16b\|u_0\|^2}{9}](\e,u^2)(t).\nonumber
\end{align}
By Theorem \ref{decay1},
$$(\e,u^2)(t)\leq e^{-\chi t}(\e,u_0^2).$$
Using this estimate, we substitute $I_1$ into \eqref{e6l1} and come
to the following inequality:
\begin{align}
&e^{\chi t}(\e,u^2)(t)+\int_0^t e^{\chi
s}\{(1+2b)(\e,u_x^2)(s)+b(\e,u_y^2)(s)\},ds\nonumber\\
&\leq C(\chi,b,\|u_0\|)(1+t)(\e,u_0^2).\nonumber
\end{align}
Since $b>0$, the proof of Proposition \ref{base} is complete.
\end{proof}
Returning to \eqref{e4l1} and using Proposition \ref{base}, we
estimate
\begin{align}
&I_1=(\chi+\frac{8b^2+16b^3}{3})(\e,u^3)(s)\leq\nonumber\\
&2(\e,|\nabla u|^2)(s)+C(\chi,b,\|u_0\|)(\e,u^2)(s).\nonumber
\end{align}
Similarly,
 \begin{align}
&I_2=4(1+4b)(\e,uu_x^2)(s)\leq \delta(\e,2u_{xx}^2+u_{xy}^2)(s)\nonumber\\
&+\big[
2b^2\delta+\frac{16(1+4b)^2\|u_0\|^2}{\delta}\big](\e,u_x^2)(s).\nonumber
\end{align}
With the help of \eqref{p2}, we find
\begin{align}
&I_3=8b(\e,uu_y^2)(s)\leq
8bC_D\|u_0\|\|e^{bx}u_y\|(s)\|e^{bx}u_y\|_{H^1(\mathcal{S})}(s)\nonumber\\
&\leq\delta(\e,2u_{xy}^2+u_{yy}^2)(s)+\big[(2b^2+1)\delta+\frac{16b^2\|u_0\|^2
C_D^2}{\delta}\big](\e, u_y^2)(s).\nonumber
\end{align}

Taking $\delta=2b$ and using Proposition \ref{base}, we obtain from
\eqref{e4l1}
$$ e^{\chi t}(\e,|\nabla u|^2)(t)\leq C(b,\chi,\|u_0\|)(1+t)(\e,
u_0^2+|u_0|^3+|\nabla u_0|^2).$$ Adding \eqref{decay}, we complete
the proof of Theorem \ref{decay2}.

\end{proof}

\section{Weak solutions}\label{weak}

Here we will prove the existence, uniqueness and continuous
dependence on initial data  as well as exponential decay
 results for weak solutions of \eqref{2.1}-\eqref{2.3} when the initial function $u_0
\in L^2(\mathcal{S}).$

\begin{teo}\label{weakexist}
Let $u_0 \in L^2 (\mathcal{S})\cap L^2_b(\mathcal{S}).$ Then for all
finite positive $T$ and $B$ there exists at least one function
$$u(x,y,t) \in L^{\infty}(0,T;L^2(\mathcal{S})),\;u_x\in
L^2(0,T;L^2(\mathcal{S}))$$ such that $$e^{bx}u \in
L^{\infty}(0,T;L^2(\mathcal{S}))\cap L^2(0,T;H^1(\mathcal{S}))$$ and
the following integral identity takes a place:
\begin{align}
&(e^{bx} u,v)(T)+\int_0^T\{-(e^{bx} u,v_t)(t)+(e^{bx}
u_x,\big[v_{xx}+(1+2b)v_x\nonumber\\&+(b+b^2)v\big])(t)-\frac{1}{2}(\e u^2,bv+v_x)(t)\nonumber\\
&+(e^{bx} u_y,bv_x+v_{xy})(t)\}\,dt=(e^{bx}
u_0,v(x,y,0)),\label{weakdefin}
\end{align}
where $v \in C^{\infty}(\mathcal{S}_T)$ is an arbitrary function.
\end{teo}
\begin{proof}
In order to justify our calculations, we must operate with
sufficiently smooth solutions $u^m(x,y,t)$. With this purpose, we
consider first initial functions $u_{0m}(x,y)$,  which satisfy
conditions of Theorem \ref{regsol}, and obtain estimates \eqref{E1},
\eqref{Ert} for functions $u^m(x,y,t)$. This allows us to pass to
the limit as $m\to \infty$ in the following identity:
\begin{align}
&(e^{bx} u^m,v)(T)+\int_0^T\{-(e^{bx} u^m,v_t)(t)+(e^{bx}
u^m_x,\big[v_{xx}+(1+2b)v_x\nonumber\\&+(b+b^2)v\big])(t)-\frac{1}{2}(\e |u^m|^2,bv+v_x)(t)\nonumber\\
&+(e^{bx} u^m_y,bv_x+v_{xy})(t)\}\,dt=(e^{bx} u_{0m},v(x,y,0))
\label{weakdefin}
\end{align}

and come to \eqref{weakdefin}.
\end{proof}

{\bf  Uniqueness of a weak solution.}

\begin{teo} \label{weakuniq}
A weak solution of Theorem \ref{weakexist} is uniquely defined.
\end{teo}

\begin{proof}
Actually, this proof is provided by Theorem \ref{uniq}. It is
sufficient to approximate the initial function $u_0\in
L^2(\mathcal{S})$ by regular functions $u_{0m}$ in the form:
$$\lim_{m\to\infty}\|u_{0m}-u_0\|=0,$$
where $u_{om}$ satisfies the conditions of Theorem \ref{regsol}.
This guarantees the existence of the unique regular solution to
(2.1)-(2.3) and allows us to repeat all the calculations which have
been done during the proof of Theorem \ref{uniq} and to come to the
following inequality:
\begin{align}
&\frac{d}{dt}(\e,z_m^2)(t)+(2+2b)(\e,z_{mx}^2)(t)+b(\e,z_{my}^2)(t)\nonumber\\
&\leq
C(b)\big[1+\|u_{1m}\|(t)^2+\|u_{2m}\|(t)^2+\|u_{1xm}\|(t)^2+\|u_{2xm}\|(t)^2\big](\e,z_m^2)(t).\nonumber
\end{align}
By the generalized Gronwall`s lemma,
\begin{align}
&(\e,z_m^2)(t)\leq exp\{\int_0^t
C(b)\big[1+\|u_{1m}\|(s)^2+\|u_{2m}\|(s)^2+\|u_{1xm}\|(s)^2\nonumber\\
&+\|u_{2xm}\|(s)^2\big]\,ds\}(\e,z_{0m}^2)(t).\nonumber
\end{align}
Functions $u_{1m}$ and $u_{2m}$ for $m$ sufficiently large satisfy
the estimate
\begin{equation}
\|u_{im}\|(t)^2+2\int_0^t \|u_{imx}\|(s)^2\,ds=\|u_{0m}\|^2\leq
2\|u_0\|^2),\quad i=1,2. \nonumber
\end{equation}

Hence,
\begin{align}
&exp\{\int_0^t
C(b)\big[1+\|u_{1m}\|(s)^2+\|u_{2m}\|(s)^2+\|u_{1xm}\|(s)^2\nonumber\\
&+\|u_{2xm}\|(s)^2\big]\,ds\}\leq C(,T,\|u_0\|).
\end{align}
 Since $e^{bx}z(x,y,t)$ is a weak limit of regular
solutions $\{e^{bx}z_m(x,y,t)\}$, then $$(\e,z^2)(t)\leq (\e,
z_m^2)(t)= 0.$$ This implies $u_1\equiv u_2$\; $a.e.$ in
$\mathcal{S}_T.$ The proof of Theorem \ref{weakuniq} is complete.
\end{proof}
\begin{rem}
Changing initial condition $z(x,y,0)\equiv 0$  for
$z(x,y,0)=z_0(x,y)\ne 0,$ and repeating the proof of Theorem
\ref{weakuniq}, we obtain  that
$$(\e,z^2)(t)\leq C(b,T,\|u_0\|)(\e,z^2_0)\quad \forall t\in(0,T).$$
This means continuous dependence of weak solutions on initial data.
\end{rem}

{\bf  Decay of weak solutions.}

\begin{teo} \label{weakdecay}
Let  $b\in(0, \frac{1}{5}[-1+\sqrt{1+\frac{5\pi^2}{4B^2}}]),\;
\|u_0\|\leq \frac{3\pi}{16B}$ and $u(x,y,t)$ be a weak solution of
\eqref{2.1}-\eqref{2.3}. Then for all finite $B>0$ the following
inequality is true:
\begin{equation}
 \|e^{bx}u\|^2(t)\leq e^{-\chi t}\|e^{bx}u_0\|^2(0),
\label{wdecay}
\end{equation}
 where $\chi=\frac{\pi^2}{20B^2}[-1+\sqrt{1+\frac{5\pi^2}{4B^2}}]$.
\end{teo}
\begin{proof}
Similarly to the proof of the uniqueness result for a weak solution,
we approximate $u_0 \in L^2(\mathcal{S})$ by sufficiently smooth
functions $u_{om}$ in order to work with regular solutions. Acting
in the same manner as by the proof of Theorem 4.1, we come to the
following inequality :
\begin{equation}
 \|e^{bx}u_m\|^2(t)\leq e^{-\chi t}\|e^{bx}u_0\|^2(0),
\label{wdecay}
\end{equation}
 where
 $$\chi=\frac{\pi^2}{20B^2}[-1+\sqrt{1+\frac{5\pi^2}{4B^2}}].$$
Since $u(x,y,t)$ is weak limit of regular solutions $\{u_m(x,y,t)\}$
then
$$(\e, u^2)(t)\leq (\e, u_m^2)(t)\leq e^{-\chi t}(\e,u_0^2).$$
The proof of Theorem \ref{weakdecay} is complete.
\end{proof}

We have in this Theorem a more strict condition  $\|u_0\|\leq
\frac{3\pi}{16B}$ instead of $\|u_0\|\leq \frac{3\pi}{8B}$ in the
case of decay for regular solution because  for weak solutions we do
not have the sharp estimate \eqref{E1l}, but only \eqref{E1}.


\begin{thebibliography}{99}


\bibitem{bona1}
\newblock C.I. Amick, J.L. Bona and M.F Schonbek,
\newblock Decay of solutions of some nonlinear wave equations,
\newblock J. of Differential Equats. 81, (1989) 1--49.



\bibitem{bona2}
\newblock J.L. Bona,  S.M. Sun  and  B.-Y. Zhang,
\newblock Nonhomogeneous problems for the Korteweg-de Vries and the Korteweg-de Vries-Burgers
equations in a quarter plane,
\newblock Ann. Inst. H. Poincar\'{e}. Anal. Non Lin\'{e}aire. 25, (2008) 1145--1185.





\bibitem{marcelo}
\newblock M.M. Cavalcanti, V.M. Domingos Cavalcanti, V. Komornik, J.N. Rodrigues,
\newblock Global well-posedness and exponential decay rates for a
KdV-Burgers equation with indefinite damping,
\newblock Ann. I. H.
Poincar\'{e} -AN(2013),
http://dx.doi.org/10.1016/j.anihpc.2013.08.003.






\bibitem{dorlar1}
\newblock G. G. Doronin and N. A. Larkin,
\newblock Stabilization of regular solutions for the Zakharov-Kuznetsov equation posed on bounded rectangles and on a strip,
\newblock arXiv:1209.5767 [math.AP] (2012).

\bibitem{fam1}
\newblock A. V. Faminskii,
\newblock The Cauchy problem for the Zakharov-Kuznetsov equation (Russian),
\newblock Differentsial'nye Uravneniya, 31 (1995) 1070--1081;
\newblock Engl. transl. in: Differential Equations 31 (1995) 1002--1012.

\bibitem{fam2}
\newblock A. V. Faminskii,
\newblock Well-posed initial-boundary value problems for the Zakharov-Kuznetsov equation,
\newblock Electronic Journal of Differential equations 127 (2008) 1--23.

\bibitem{fam3}
\newblock A. V. Faminskii,
\newblock An initial-boundary value problem in a strip for a two-dimensional equation of Zakharov-Kuznetsov type,
\newblock arXiv:1312.4444v1 [math.AP] 16 Dec 2013.

\bibitem{fambayk}
\newblock A. V. Faminskii and E. S. Baykova,
\newblock On initial-boundary value problems in a strip for the generalized two-dimensional Zakharov-Kuznetsov equation,
\newblock arXiv:1212.5896v2 [math.AP] 15 Jan 2013.



\bibitem{familark}
\newblock A.V.Faminskii and N.A. Larkin,
\newblock Initial-boundary value problems for quasilinear dispersive
equations posed on a bounded interval,
\newblock Electron. J.
Differ. Equations. (2010) 1--20.





\bibitem{farah}
\newblock L. G. Farah, F. Linares and A. Pastor,
\newblock A note on the 2D generalized Zakharov-Kuznetsov equation: Local, global, and scattering results,
\newblock J. Differential Equations 253 (2012) 2558–-2571.

\bibitem{kato}
\newblock T. Kato,
\newblock On the Cauchy problem for the (generalized) Korteweg-de Vries equations,
\newblock Advances in Mathematics Suplementary Studies, Stud. Appl. Math. 8 (1983) 93--128.

\bibitem{ponce}
\newblock C.E. Kenig, G. Ponce and L. Vega,
\newblock Well-posedness and
scattering results for the generalized Korteweg-de Vries equation
and the contraction principle,
\newblock Commun. Pure Appl. Math. 46
No 4, (1993) 527--620.




\bibitem{lady}
\newblock O.A. Ladyzhenskaya, V.A. Solonnikov, and  N.N. Uraltseva,
\newblock Linear and
Quasilinear Equations of Parabolic Type,
\newblock American
Mathematical Society. Providence. Rhode Island, 1968.


\bibitem{lartron}
\newblock N.A. Larkin and E. Tronco,
\newblock Regular solutions of the 2D Zakharov-Kuznetsov equation on a half-strip,
\newblock J. Differential Equations 254 (2013) 81–-101.

\bibitem{lar1}
\newblock N.A. Larkin,
\newblock Exponential decay of the $H^1$-norm for the 2D Zakharov-Kuznetsov equation on a half-strip,
\newblock J. Math. Anal. Appl. 405 (2013) 326--335.



\bibitem{pastor1}
\newblock F. Linares and A. Pastor,
\newblock Local and global well-posedness for the 2D generalized Zakharov-Kuznetsov equation,
\newblock J. Funct. Anal. 260 (2011) 1060--1085.

\bibitem{pastor2}
\newblock F. Linares, A. Pastor and J.-C. Saut,
\newblock Well-posedness for the ZK equation in a cylinder and on the background of a KdV Soliton,
\newblock Comm. Part. Diff. Equations 35 (2010) 1674--1689.


\bibitem{sautlin}
\newblock F. Linares and J.-C. Saut,
\newblock The Cauchy problem for the 3D Zakharov-Kuznetsov equation,
\newblock Disc. Cont. Dynamical Systems A 24 (2009) 547--565.



\bibitem{ribaud}
\newblock F. Ribaud, S. Vento,
\newblock Well-posedness results for the three-dimensional Zakharov-Kuznetsov equation,
\newblock SIAM J. Math. Anal. 44 (2012) 2289--2304.

\bibitem{rozan}
\newblock L. Rosier and B.-Y. Zhang,
\newblock Control and stabilization of the KdV equation: recent progress,
\newblock J. Syst. Sci. Complexity 22 (2009) 647--682.




\bibitem{temam1}
\newblock J.-C. Saut and R. Temam,
\newblock An initial boundary-value problem for the Zakharov-Kuznetsov equation,
\newblock Advances in Differential Equations 15 (2010) 1001--1031.

\bibitem{temam2}
\newblock J.-C. Saut, R. Temam and C. Wang,
\newblock An initial and boundary-value problem for the Zakharov-Kuznetsov equation in a bounded domain,
\newblock J. Math. Phys. 53 115612 (2012).


\bibitem{zk}
\newblock V. E. Zakharov and E. A. Kuznetsov,
\newblock On three-dimensional solitons,
\newblock Sov. Phys. JETP 39 (1974) 285--286.






\end{thebibliography}
\end{document}